\documentclass{amsart}
\usepackage{amssymb,latexsym}
\usepackage{graphicx}
\usepackage{color}
\usepackage{enumitem}
\usepackage{cancel}
\usepackage{wasysym}
\usepackage[foot]{amsaddr}

\newtheorem{theorem}{Theorem}[section]
\newtheorem{definition}[theorem]{Definition}
\newtheorem{remark}[theorem]{Remark}
\newtheorem{lemma}[theorem]{Lemma}
\newtheorem{corollary}[theorem]{Corollary}

\newcommand{\abs}[1]{\lvert#1\rvert}
\newcommand{\K}{\mathbb K}

\newcommand{\F}{\mathbb F}
\newcommand{\bu}{\mathbf u}

\newcommand{\cG}{\mathcal G}

\newcommand{\cX}{\mathcal X}
\newcommand{\cY}{\mathcal Y}
\newcommand{\cC}{\mathcal C}
\newcommand{\cH}{\mathcal H}
\newcommand{\cI}{I}

\newcommand{\ord}{\mathrm{ord}}
\newcommand{\PG}{\mathrm{PG}}

\def\zhou#1 {\fbox {\footnote {\ }}\ \footnotetext { From Yue: {\color{red}#1}}}
\def\daniele#1 {\fbox {\footnote {\ }}\ \footnotetext { From Daniele: {\color{blue}#1}}}

\begin{document}
\title{Exceptional scattered polynomials}
\author{Daniele Bartoli\textsuperscript{\,1}}
\author{Yue Zhou\textsuperscript{\,2,$\dagger$}}
\address{\textsuperscript{1}Department of Mathematics and Computer Science, University of Perugia, 06123 Perugia, Italy}
\email{daniele.bartoli@unipg.it}
\address{\textsuperscript{2}College of Science, National University of Defense Technology, 410073 Changsha, China}
\address{\textsuperscript{$\dagger$}Corresponding Author}
\email{yue.zhou.ovgu@gmail.com}
\begin{abstract}
Let $f$ be an $\F_q$-linear function over $\F_{q^n}$. If the $\F_q$-subspace $U= \{ (x^{q^t}, f(x)) : x\in \F_{q^n}  \}$ defines a maximum scattered linear set, then we call $f$ a scattered polynomial of index $t$. As these polynomials appear to be very rare, it is natural to look for some classification of them. We say a function $f$ is an exceptional scattered polynomial of index $t$ if the subspace $U$ associated with $f$ defines a maximum scattered linear set in $\PG(1, q^{mn})$ for infinitely many $m$. Our main results are the {complete} classifications of exceptional scattered monic polynomials of index $0$ (for $q>5$) and of index $1$. The strategy applied here is to convert the original question into a special type of algebraic curves and then to use the intersection theory and the Hasse-Weil theorem to derive contradictions.
\end{abstract}
\maketitle
\section{Introduction}
Let $q$ be a prime power and $r,n\in \mathbb{N}$. Let $V$ be a vector space of dimension $r$ over $\F_{q^n}$. For any $k$-dimensional $\F_q$-vector subspace $U$ of $V$, the set $L(U)$ defined by the nonzero vectors of $U$ is called an $\F_q$-\emph{linear set} of $\Lambda=\PG(V, q^n)$ of \emph{rank} $k$, i.e.
\[ L(U)=\{\langle \bu \rangle_{\F_{q^n}}: \bu \in U\setminus \{\mathbf{0} \}  \}.\]
It is notable that the same linear set can be defined by different vector subspaces. Consequently, we always consider a linear set and the vector subspace defining it in pair. 

Let $\Omega=\PG(W,\F_{q^n})$ be a subspace of $\Lambda$ and $L(U)$ an $\F_q$-linear set of $\Lambda$. We say that $\Omega$ has \emph{weight} $i$ in $L(U)$ if $\dim_{\F_q}(W\cap U)=i$. Thus a point of $\Lambda$ belongs to $L(U)$ if and only if it has weight at least $1$. Moreover, for any $\F_q$-linear set $L(U)$ of rank $k$, 
\[|L_U|\leq \frac{q^{k}-1}{q-1}.\]
When the equality holds, i.e.\ all the points of $L(U)$ have weight $1$, we say $L(U)$ is \emph{scattered}. A scattered $\F_q$-linear set of highest possible rank is called a \emph{maximum scattered} $\F_q$-\emph{linear set}. See \cite{blokhuis_scattered_2000} for the possible ranks of maximum scattered linear sets.

Maximum scattered linear sets have various applications in Galois geometry, including blocking sets \cite{ball_linear_2000,lunardon_linear_k-blocking_2001,lunardon_blocking_2000}, two-intersection sets \cite{blokhuis_scattered_2000,blokhuis_two-intersection_2002}, finite semifields \cite{cardinali_semifield_2006,ebert_infinite_2009,lunardon_maximum_scattered_2014,marino_towards_2011}, translation caps \cite{bartoli_maximum_2017}, translation hyperovals \cite{durante_hyperovals_2017}, etc. For more applications and related topics, see \cite{polverino_linear_2010} and the references therein. For recent surveys on linear sets and particularly on the theory of scattered spaces, see \cite{lavrauw_scattered_2016,lavrauw_field_reduction_2015}.

In this paper, we are interested in maximum scattered linear sets in $\PG(1, q^n)$. Let $f$ be an $\F_q$-linear function over $\F_{q^n}$ and
\begin{equation}\label{eq:U_(x,f)}
	U= \{ (x, f(x)) : x\in \F_{q^n}  \}.
\end{equation}
Clearly $U$ is an $n$-dimensional $\F_q$-subspace of $\F_{q^n}$ and $f$ can be written as a $q$-polynomial $f(X)=\sum a_i X^{q^i} \in \F_{q^n}[X]$. It is not difficult to show that a necessary and sufficient condition for $L(U)$ to define a maximum scattered linear set in $\PG(1, q^n)$ is
\begin{equation}\label{eq:scattered_poly}
	\frac{f(x)}{x} = \frac{f(y)}{y} ~\text{ if and only if }~ \frac{y}{x}\in \F_q, \quad\text{ for }x,y\in \F_{q^n}^*.
\end{equation}
In \cite{sheekey_new_2016}, such a $q$-polynomial is called a \emph{scattered polynomial}.

Two linear sets $L(U)$ and $L(U')$ in $\PG(2,q^n)$ are \emph{equivalent} if there exists an element of $\mathrm{P\Gamma L}(2, q^n)$ mapping $L(U)$ to $L(U')$. It is obvious that if $U$ and $U'$ are equivalent as $\F_{q^n}$-spaces, then $L(U)$ and $L(U')$ are equivalent. However, the converse is not true in general. For recent results on the equivalence and classification of linear sets, we refer to \cite{csajbok_classes_2017,csajbok_maximum_2017,csajbok_equivalence_2016}.

There is a very interesting link between maximum scattered linear sets and the so called maximum rank distance (MRD for short) codes \cite{csajbok_maximum_2017}. In particular, a scattered polynomial over $\F_{q^n}$ defines an MRD code in $\F_q^{n\times n}$ of minimum distance $n-1$.

Let us give a brief introduction to MRD codes. The \emph{rank metric} on the $\F_q$-vector space $\F_q^{m\times n}$ is defined by
\[
d(A,B)=\mathrm{rank}(A-B) \,\, \text{for} \,\, A,B\in \F_q^{m\times n}.
\]
We call a subset of $\F_q^{m\times n}$ equipped with the rank metric a \emph{rank-metric code}. For a rank-metric code $\cC$ containing at least two elements, its \emph{minimum distance} is given by
$$d(\cC)=\min_{A,B\in \cC, A\neq B} d(A,B).$$ 
When $\cC$ is an $\F_q$-subspace of $\F_q^{m\times n}$, we say that $\cC$ is an \emph{$\F_q$-linear} code of dimension $\dim_{\F_q}(\cC)$. Under the assumption that $m\leq n$, it is well known (and easily verified) that every rank-metric code $\cC$ in $\F_q^{m\times n}$ with minimum distance $d$ satisfies
\[
\abs{\cC}\le q^{n(m-d+1)}.
\]
In case of equality, $\cC$ is called a \emph{maximum} rank-metric code, or \emph{MRD code} for short. MRD codes have been studied since the 1970s and have seen much interest in recent years due to an important application in the construction of error-correcting codes for random linear network coding~\cite{koetter_coding_2008}.

When the minimum distance $d=m=n$, an MRD code in $\F_q^{m\times n}$ is exactly a spreadset which is equivalent to a finite (pre)quasifield of order $q^n$. Many essentially different families of finite quasifields are known \cite{johnson_handbook_2007}, which yield many inequivalent MRD codes in $\F_q^{n\times n}$ with minimum distance $n$. In contrast, it appears to be much more difficult to obtain inequivalent MRD codes in $\F_q^{m\times n}$ with minimum distance strictly less than $n$ (recall that we assume $m\le n$).

A canonical construction of MRD codes of any minimum distance $d$ was given by Delsarte~\cite{delsarte_bilinear_1978}. This construction was rediscovered by Gabidulin~\cite{gabidulin_MRD_1985} and later generalized by Kshevetskiy and Gabidulin~\cite{kshevetskiy_new_2005}. Today it is customary to call the codes in this generalized family the \emph{Gabidulin codes}.

In recent years, an increased interest emerged concerning the question as to whether Gabidulin codes are unique at least for certain parameter sets, or if not, what other constructions can be found. Partial answers were given recently by Horlemann-Trautmann and Marshall~\cite{horlemann-trautmann_new_2015}, who showed indeed that Gabidulin codes are unique among $\F_q$-linear MRD codes for certain parameters. On the other hand there are several recent constructions of MRD codes, which were proved to be inequivalent to Gabidulin codes~\cite{cossidente_non-linear_2016,csajbok_maximum_2017,durante_nonlinear_MRD_2017,horlemann-trautmann_new_2015,lunardon_generalized_2015,neri_genericity_2017,sheekey_new_2016}. Very recently in \cite{schmidt_number_MRD_2017}, it was showed that the family of Gabidulin codes in $\F_q^{m\times n}$ already contains a huge subset of pairwise inequivalent MRD codes, provided that $2\le m\le n-2$.

However, for $m=n$ and $d<n$, there are only a few known constructions of MRD codes including Gabidulin codes \cite{delsarte_bilinear_1978,gabidulin_MRD_1985,kshevetskiy_new_2005}, twisted Gabidulin codes \cite{sheekey_new_2016}, the nonlinear MRD codes found by Cossidente, Marino and Pavese \cite{cossidente_non-linear_2016} with minimum distance $2$ which were later generalized by Durante and Siciliano \cite{durante_nonlinear_MRD_2016}, and the additive MRD codes obtained by Otal and \"Ozbudak \cite{otal_additive_2016}.

In particular, given a scattered polynomial $f$ over $\F_{q^n}$, an MRD codes can be defined by the following set of $\F_q$-linear maps
\begin{equation}\label{eq:MRD_scattered}
	 C_f :=\{ax+bf(x) : a, b\in \F_{q^n}\}. 
\end{equation}
To show that \eqref{eq:MRD_scattered} defines an MRD codes, we only have to prove that $ax+bf(x)$ has at most $q$ roots for each $a,b\in \F_{q^n}$ with $ab\neq 0$, which is equivalent to \eqref{eq:scattered_poly}.

It is worth pointing out that the MRD code defined by \eqref{eq:MRD_scattered} is $\F_{q^n}$-linear. Using the terminology in \cite{lunardon_kernels_2016}, one of its nuclei is $\F_{q^n}$. The equivalence problem of $\F_{q^n}$-linear MRD codes is slightly easier to handle compared with other MRD codes; see \cite{morrison_equivalence_2014}. It can be easily proved that for two given scattered polynomials $f$ and $g$, if they define two equivalent MRD codes, then the two associated maximum scattered linear sets are also equivalent. However the converse statement is not true in general; see \cite{csajbok_classes_2017,sheekey_new_2016}.

To the best of our knowledge, up to the equivalence of the associated MRD codes, all constructions of scattered polynomials for arbitrary $n$ can be summarized as one family
\begin{equation}\label{eq:the_one}
	f(x)=\delta x^{q^s} + x^{q^{n-s}},
\end{equation}
where $s$ satisfies $\gcd(s,n)=1$ and $\mathrm{Norm}_{\F_{q^n}/\F_q}(\delta)=\delta^{(q^n-1)/(q-1)}\neq 1$. 

When $\delta=0$ and $n-s=1$, $f$ defines the maximum scattered $\F_q$-linear set in $\PG(1, q^n)$ found by Blokhuis and Lavrauw \cite{blokhuis_scattered_2000}. In fact, no matter which value $s$ takes, $f(x)=x^{q^s}$ defines the same maximum scattered $\F_q$-linear set. However, the MRD codes associated with $x^{q^s}$ and $x^{q^t}$ are inequivalent if and only if $s\not\equiv \pm t \pmod{n}$.

When $\delta\neq 0$, $f$ defines the MRD codes constructed by Sheekey in \cite{sheekey_new_2016} and the equivalence problem was completely solved in \cite{lunardon_generalized_2015}. In particular, when $s=1$, the associated maximum scattered $\F_q$-linear set in $\PG(1, q^n)$ was found  by Lunardon and Polverino \cite{lunardon_blocking_2001}. In \cite{csajbok_classes_2017}, it is claimed that for different $s$ the associated linear sets can be inequivalent.

Besides the family of scattered polynomials defined in \eqref{eq:the_one}, very recently, Csajb\'ok, Marino, Polverino and Zanella find another new family of MRD codes which are of the form 
\begin{equation}\label{eq:the_two}
f(x)=\delta x^{q^s} + x^{q^{n/2+s}},
\end{equation}
for $n=6, 8$ and some $\delta\in\F_{q^n}^*$; see \cite{csajbok_newMRD_2017}.

As scattered polynomials appear to be very rare, it is natural to look for some classification of them. Given an integer $0\le t\le n-1$ and a $q$-polynomial $f$ whose coefficients are in $\F_{q^n}$, if 
\begin{equation}\label{eq:main_target}
	U_m= \{(x^{q^t}, f(x)): x\in \F_{q^{mn}}  \}
\end{equation}
defines a maximum scattered linear set in $\PG(1,q^{mn})$ for infinitely many $m$, then we call $f$ an \emph{exceptional scattered polynomial of index} $t$. {In particular, if $U_1$ is maximum scattered, then we say $f$ is a scattered polynomial over $\F_{q^n}$ of index $t$. } 

In \eqref{eq:main_target}, our linear set is slightly different from the one defined by \eqref{eq:U_(x,f)}, because we want to view the unique known family \eqref{eq:the_one} as exceptional ones: Taking $t=s$, from \eqref{eq:the_one} we get
\[\{(x^{q^s},x + \delta x^{q^{2s}} ): x\in \F_{q^{mn}}   \}  \]
which defines a maximum scattered linear set for all $mn$ satisfying $\gcd(mn,s)=1$. This means $x + \delta x^{q^{2s}}$ is exceptional of index $s$.

Assume that $U_m$ given by \eqref{eq:main_target} defines a maximum scattered linear set for some $m$. Now we want to normalize our research objects to exclude some obvious cases.
\begin{itemize}
	\item Without loss of generality, we assume that the coefficient of $X^{q^t}$ in $f(X)$ is always $0$.
	\item When $t>0$, we assume that the coefficient of $X$ in $f(X)$ is nonzero; otherwise let $t_0=\min\{i: a_i\neq 0\}$ and it is equivalent to consider
	\[ \left\{\left(x^{q^{t-t_0}}, \sum_{i=t_0}^{n-1}a_i^{q^{n-t_0}}x^{q^{i-t_0}}\right): x\in \F_{q^{mn}}  \right\} \]
	instead of $U_m$.
	\item We assume that $f(X)$ is monic.
\end{itemize}

After the normalization mentioned above, the main results concerning exceptional scattered polynomials in this paper are as follows.
\begin{enumerate}
	\item For $q>5$, $X^{q^k}$ is the unique exceptional scattered monic polynomial of index $0$.
	\item The only exceptional scattered monic polynomials $f$ of index $1$ over $\F_{q^n}$ are $X$ and $bX + X^{q^2}$ where $b\in \F_{q^n}$ satisfying $\mathrm{Norm}_{q^n/q}(b)\neq 1$. In particular, when $q=2$, $f(X)$ must be $X$.
\end{enumerate}

{To prove these results, the brief idea  is to convert the original question into a special type of algebraic curves and then to use some classical approaches such as the intersection theory and the Hasse-Weil theorem to get contradictions.}

There are several famous functions defined over finite fields which are also quite rare and similar classification of them have been obtained. For instance, to classify functions on $\F_{p^n}$ that are almost perfect nonlinear for infinitely many $n$, in particular for monomial functions, Janwa, McGuire and Wilson \cite{janwa_double-error-correcting_1995} proposed to use ideas form algebraic geometry. Later these ideas were developed by Jedlicka \cite{jedlicka_apn_2007} and Hernando and McGuire \cite{hernando_proof_2011}. The same approach has been applied in \cite{hernando_proof_2012} to prove a conjecture on monomial hyperovals and in \cite{leducq_functions_PN_2015} to get partial results towards the classification of monomial planar functions for infinitely many $n$, which was later completely solved by Zieve~\cite{zieve_planar_2015} by using the classification of indecomposable exceptional (permutation) polynomials. Similar results and approaches can also be found in \cite{caullery_large_class_2014,caullery_classification_2015,caullery_exceptional_2016,rodier_functions_APN_2011,schmidt_planar_2014}.

In this paper, we follow the main idea of the algebraic curve approach applied in the references listed above. However, to obtain a better estimation of the intersection number of singular points of two curves, our approach requires the use of branches and the local quadratic transformations of a plane curve. As far as we know, these tools have not been applied in the previous works for classifying polynomials over finite fields.

\section{Preliminaries}
For a given integer $0\le t\le n-1$ and a $q$-polynomial $f(X)=\sum_{i=0}^{n-1}a_iX^{q^i}$ whose coefficients are in $\F_{q^n}$, let
\begin{equation}\label{eq:U_(xt,f)}
U= \{(x^{q^t}, f(x)): x\in \F_{q^{n}}\}.
\end{equation}

It is easy to show that one necessary and sufficient condition for \eqref{eq:U_(xt,f)} defining a maximum scattered $\F_q$-linear set in $\PG(1,q^n)$ is
\[\frac{f(x)}{x^{q^t}} = \frac{f(y)}{y^{q^t}} ~\text{ if and only if }~ \frac{y}{x}\in \F_q, \quad\text{ for }x,y\in \F_{q^n}^*,\]
from which the following lemma follows.
\begin{lemma}\label{le:link}
	The vector space $U= \{(x^{q^t}, f(x)): x\in \F_{q^{n}}\}$ defines a maximum scattered linear set $L(U)$ in $\PG(1, q^n)$ if and only if the curve defined by
	\begin{equation}\label{eq:curve_condition}
	\frac{f(X)Y^{q^t} - f(Y)X^{q^t}}{X^qY-XY^q}
	\end{equation}
	in $\PG(2,q^n)$ contains no affine point $(x,y)$ such that $\frac{y}{x}\notin\F_{q}$.
\end{lemma}

By Lemma \ref{le:link}, to prove a polynomial $f$ does not define a maximum scattered linear set by \eqref{eq:U_(xt,f)}, we only have to show that the curve defined by \eqref{eq:curve_condition} has at least one affine point $(x,y)$ such that $y/x\in \F_{q^n}\setminus \F_q$. To investigate these algebraic curves, we need to introduce several important concepts and results.

Let $F$ be a polynomial defining an affine plane curve $\mathcal{F}$, let $P=(u,v)$ be a point in the plane, and write
\[
F(X+u,Y+v)=F_0(X,Y)+F_1(X,Y)+F_2(X,Y)+\cdots,
\]
where $F_i$ is either zero or homogeneous of degree $i$. The \emph{multiplicity} of $\mathcal{F}$ at $P$, written as $m_P(\mathcal{F})$, is the smallest integer $m$ such that $F_m\ne 0$ and $F_i=0$ for $i<m$; the polynomial $F_m$ is the \emph{tangent cone} of $\mathcal{F}$ at $P$. A divisor of the tangent cone is called a \emph{tangent} of $\mathcal{F}$ at $P$. The point $P$ is on the curve $\mathcal{F}$ if and only if $m_P(\mathcal{F})\ge 1$. If $P$ is on $\mathcal{F}$, then $P$ is a \emph{simple} point of $\mathcal{F}$ if $m_P(F)=1$, otherwise $P$ is a \emph{singular} point of $\mathcal{F}$. It is possible to define in a similar way the multiplicity of an ideal point of $\mathcal{F}$, that is a point of the curve lying on the line at infinity.

Given two plane curves $\mathcal{A}$ and $\mathcal{B}$ and a point $P$ on the plane, the \emph{intersection number} $I(P, \mathcal{A} \cap \mathcal{B})$ of $\mathcal{A}$ and $\mathcal{B}$ at the point $P$ is defined by seven axioms. We do not include its precise and long definition here. For more details, we refer to \cite{fulton_algebraic_1989} and \cite{hirschfeld_curves_book_2008} in which the intersection number is defined equivalently in terms of local rings and in terms of resultants, respectively.

For intersection number, we need the following two classical results which can be found in most of the textbooks on algebraic curves.
\begin{lemma}\label{le:ordinary_singular}
	Let $\mathcal{A}$ and $\mathcal{B}$ be two plane curves and let $A$ and $B$ be the polynomials associated with $\mathcal{A}$ and $\mathcal{B}$ respectively.  For any affine point $P$, the intersection number satisfies the inequality
	\[ I(P, \mathcal{A}\cap \mathcal{B})\ge m_P(A) m_P(B), \]
	with equality if and only if the tangents at $P$ to $\mathcal{A}$ are all distinct from the tangents at $P$ to $\mathcal{B}$.
\end{lemma}

\begin{theorem}[B\'ezout's Theorem]\label{th:bezout}
	Let $\mathcal{A}$ and $\mathcal{B}$ be two projective plane curves over an algebraically closed field $\K$, having no component in common. Let $A$ and $B$ be the polynomials associated with $\mathcal{A}$ and $\mathcal{B}$ respectively. Then
	\[
	\sum_P I(P, \mathcal{A}\cap \mathcal{B})=(\deg A)(\deg B),
	\]
	where the sum runs over all points in the projective plane $\PG(2,\K)$.
\end{theorem}

We denote by $\mathcal{F}(\mathbb{K})$ the set of $\mathbb{K}$-rational points of $\mathcal{F}$, that is the points of $PG(2,\mathbb{K})$ belonging to the curve $\mathcal{F}$.

We also need the following results to estimate the intersection number, which is not difficult to prove (see Janwa, McGuire, and Wilson~\cite[Proposition 2]{janwa_double-error-correcting_1995}).
\begin{lemma}
	\label{le:intersection_number_m_m1_coprime}
	Let $F$ be  a polynomial in $\F_q[X,Y]$ and suppose that $F=AB$. Let $P=(u,v)$ be a point in the affine plane $\mathrm{AG}(2,q)$ and write
	\[
	F(X+u,Y+v)=F_m(X,Y)+F_{m+1}(X,Y)+\cdots,
	\]
	where $F_i$ is zero or homogeneous of degree $i$ and $F_m\ne 0$. Let~$L$ be a linear polynomial and suppose that $F_m=L^m$ and $L\nmid F_{m+1}$. Then $I(P, \mathcal{A}\cap \mathcal{B})=0$, where $\mathcal{A}$ and $\mathcal{B}$ are the curves defined by $A$ and $B$ respectively.
\end{lemma}

The next result was proved in \cite[Lemma 4.3]{schmidt_planar_2014} for $q$ even case. Actually it still holds when $q$ is odd and its proof is almost the same.
\begin{lemma}
	\label{le:intersection_number_linear_term}
	Let $F$ be a polynomial in $\F_q[X,Y]$ and suppose that $F=AB$. Let $P=(u,v)$ be a point in the affine plane $\mathrm{AG}(2,q)$ and write
	\[
	F(X+u,Y+v)=F_m(X,Y)+F_{m+1}(X,Y)+\cdots,
	\]
	where $F_i$ is zero or homogeneous of degree $i$ and $F_m\ne 0$. Let $L$ be a linear polynomial and suppose that $F_m=L^m$, $L \mid F_{m+1}$, $L^2 \nmid F_{m+1}$. Then $I(P, \mathcal{A}\cap \mathcal{B})=0$ or $m$, where $\mathcal{A}$ and $\mathcal{B}$ are the curves defined by $A$ and $B$ respectively.
\end{lemma}

In the second part of our main results in Section \ref{sec:main}, we need to use the branches of a plane curve to estimate the intersection numbers. Let $\K$ be a field and let $\K[[t]]$ denote the ring of formal power series over $\K$. We define $\K((t))$ by
\begin{align*}
	\K((t)) &= \{F/G : F,G\in \K[[t]], G\neq 0 \}\\
	&=\left\{ \frac{t^sE_1(t)}{t^rE_2(t)}: E_1(t), E_2(t) \text{ are invertible in } \K[[t]]  \right\}.
\end{align*}
It is not difficult to verify that $\K((t))$ is a field, which is called the \emph{field of rational functions of the formal power series} in the indeterminate $t$. For an element $F=t^m E(t)\in \K((t))$ with $E(t)$ invertible in $\K[[t]]$, its \emph{order} $m$ is denoted by $\ord_t F$.

A \emph{branch representation} is just a point $(x(t), y(t), z(t))$ in the plane $\PG(2, \K((t)))$ not belonging to $\PG(2,\K)$.  For a branch representation $(x(t), y(t), z(t))$, if the order of one of its components is zero and the other two are non-negative, then it is called a \emph{special} branch representation. In this case, the point $(x(0), y(0), z(0))$ is called the \emph{center} of the branch representation.

A \emph{branch} is just the equivalence class of primitive branch representations, and a \emph{branch} of a plane curve defined by $F(X,Y,Z)\in \F[X,Y,Z]$ is a branch whose representations are zeros of $F(X,Y,Z)$. For more details on the definition of branches and related concepts, we refer to \cite[Chapter 4]{hirschfeld_curves_book_2008}.

Branches constitute an important tool for the local study of algebraic curves.
\begin{definition}\label{def:intersection_multi}
	Let $\cC$ be a plane curve defined by a homogeneous polynomial $F(X,Y,Z)$ and let $\gamma$ be a branch centered at the point $P$. If $(\xi_0(t),\xi_1(t),\xi_2(t))$ is a representation of $\gamma$ in special coordinates, then the \emph{intersection multiplicity} is
	\[I(P, \cC \cap \gamma) = \left\{
	\begin{array}{ll}
	\ord_tF(\xi_0(t),\xi_1(t),\xi_2(t)) & \text{ if }\gamma\not\in \cC,  \\ 
	\infty & \text{ if }\gamma\in \cC.
	\end{array} 
	\right.  \]
\end{definition}

Clearly, the center $P\in \PG(2,\K)$ of a branch $\gamma$ such that $I(P, \cC \cap \gamma)>0$  is a point of the curve $\cC$.

The intersection multiplicity depends only on $\gamma$, not on the chosen representation. It appears that we are abusing notations, because the intersection number and the intersection multiplicity are both denoted by $I$. The next result shows the close link between these two concepts, from which the reason that we use the same notation to denote them becomes quite obvious and natural.
\begin{theorem}\cite[Theorem 4.36 (ii)]{hirschfeld_curves_book_2008}\label{th:intersection_branch}
	If $P$ is a singular point of the irreducible curve $\mathcal{F}$ and $\cG$ is a plane curve not containing $\mathcal{F}$ as a component, then
	\[ I(P, \cG \cap \mathcal{F}) = \sum_{\gamma} I(P, \cG \cap \gamma), \]
	where $\gamma$ runs over all branches of $\mathcal{F}$ centered at $P$.
\end{theorem}
There is exactly one branch of $\cC$ centered at every non-singular point of $\cC$, but an ordinary $r$-fold point is the centre of exactly $r$ branches of $\cC$. In general, an $r$-fold point of $\cC$ is the centre of at least one and at most $r$ branches of $\cC$.

The following two lemmas are required in our proof. They can be found in \cite[Theorem 4.6, Theorem 4.37]{hirschfeld_curves_book_2008}.
\begin{lemma}\label{le:branch_(t,t+..)}
	Let $F(X,Y)\in \K[X,Y]$ with $F(0,0)=0$. Let $\partial F/\partial Y\neq 0$ at $(0,0)$. Then there is a unique formal power series $c_1X + c_2X^2 + \cdots \in \K[[X]]$ such that
	\[F(X, c_1X + c_2X^2 + \cdots)=0.  \] 
\end{lemma}

\begin{lemma}\label{le:2_irre_curve_no_branch}
	Two distinct irreducible plane curves have no branches in common.
\end{lemma}

	The study of branches centered at a singular point of a plane curve can be helped by local quadratic transformations. For the points not on the line $X=0$, a local quadratic transformation is given by
	\begin{align}
	\nonumber	\sigma(X,Y) &= (X' ,Y'),\\
	\label{eq:def_local_trans}	X' &= X,\\
	\nonumber	Y' &=Y/X.
	\end{align}
	For those points on $X=0$, $Y_\infty=(0,1,0)$ is taken to be the image of them. Let $(\xi(t), \eta(t))$ be a branch representation centered at the origin and $\xi(t)\neq0$. Then $(\xi(t), \eta(t)/\xi(t))$ is the image of the branch $(\xi(t), \eta(t))$ under $\sigma$. 
	
	The \emph{geometric transform} of a curve $\mathcal{F}$ defined by $F(X,Y)=0$ with the origin as a point of multiplicity $r$ is the curve $\mathcal{F}'$ given by
	\[ F'(X,Y)= F(X,XY)/X^r.  \]
	
	The following result can be found in \cite[Theorem 4.44]{hirschfeld_curves_book_2008}. 
	\begin{theorem}\label{th:geometrc_transform}
		If $\mathcal{F}$ is a plane curve such that the line $X=0$ is not tangent to it at the origin and $\mathcal{F}'$ is the geometric transform of $\mathcal{F}$ by the local quadratic transformation $\sigma$ as in \eqref{eq:def_local_trans}, then there exists a bijection between the branches of $\mathcal{F}$ centered at the origin and the branches of $\mathcal{F}'$ centered at an affine point on $X=0$.
	\end{theorem}

After introducing the intersection numbers of two curves, we also need a bit more concepts and results. An algebraic curve defined over $\K$  is \emph{absolutely irreducible}  if the associated polynomial is irreducible over every algebraic extension of $\K$. An absolutely irreducible $\mathbb{K}$-rational component of a curve $\mathcal{C}$, defined by the polynomial $F$, is simply an absolutely irreducible curve such that the associated polynomial has coefficients in $\K$ and it is a factor of $F$.

The original Hasse-Weil bound is given in terms of genera of curves. Here we only need a weak version of it.
\begin{theorem}[Hasse-Weil Theorem]\label{th:HW-bound}
		For an absolutely irreducible curve $\cC$ in $\PG(2,q)$, then
		\[ \abs{ \# \cC(\F_q) - q-1 } \le (d-1)(d-2)\sqrt{q},  \]
		where $d$ is the degree of the defining polynomial for $\cC$.
\end{theorem}

We need one further standard result to study non-absolutely irreducible curves later (see Hernando and McGuire~\cite[Lemma~10]{hernando_proof_2011}, for example).
\begin{lemma}
	\label{le:splitting_of_irreducible_polys}
	Let $F\in\F_q[X_1,\dots,X_m]$ be a polynomial of degree $d$, irreducible over~$\F_q$. Then there exists a natural number $s\mid d$ such that, over its splitting field, $F$ splits into $s$ absolutely irreducible polynomials, each of degree $d/s$.
\end{lemma}

\section{Main Theorems}\label{sec:main}

In this section, we investigate exceptional scattered polynomials of index $0$ and $1$. As the proof of the main results are quite different, we separate the rest of this section into two parts. To make our calculation slightly convenient, we always assume that the coefficient of the lowest-degree term, instead of the leading term,  is $1$.

\subsection{Exceptional scattered polynomials of index $0$}
Let us define
\begin{equation}\label{eq:f_main_0}
	f(X)=X^{q^i} + \varphi(X) + bX^{q^k},
\end{equation} 
where $i\ge 1$, $b\neq 0$, $\varphi(X)=\sum _{j=i+1}^{k-1} \alpha_j X^{q^{j}}$ with $\alpha_j\in \F_{q^n}$.

First we show that the curve derived from $f$ as in \eqref{eq:curve_condition} contains at least one absolutely irreducible $\mathbb{F}_{q^n}$-rational component.

\begin{theorem}\label{th:main_0}
	Let $\cC$ be the curve defined by the polynomial
	$$F(X,Y)=\frac{(X^{q^{i}} + \varphi(X)+bX^{q^k})Y-X(Y^{q^{i}}+\varphi(Y)+b Y^{q^k})}{X^{q}Y-XY^{q}}.$$
	For $f$ defined by \eqref{eq:f_main_0}, we assume that $\ker(f)= q^{\ell-i}$.
	If $\gcd(k,n)=1$ and 
	\[q^{\ell+i}+q^\ell-q^{2i}-q^i +\frac{(q^{i}-q)^2}{4}< \frac{2}{9}(q^{k}-q)^2,\]
	then $\cC$ has at least one $\F_{q^n}$-rational component.
\end{theorem}
\begin{proof}
	Suppose, by way of contradiction, that $\cC$ has no absolutely irreducible components over $\F_{q^n}$. The rest of the proof is divided into several steps: first, we list all singular points of $\cC$. Second, we assume that $\cC$ splits into two components $\cX$ and $\cY$ sharing no common irreducible component, and we prove an upper bound on the total intersection number of $\cX$ and $\cY$. Then, under the assumption that $\cC$ has no absolutely irreducible components over $\F_{q^n}$, we decompose $F(X,Y) = A(X,Y) B(X,Y)$ and obtain a lower bound on $(\deg A) (\deg B)$. Finally, by using B\'ezout's theorem (see Theorem \ref{th:bezout}), we get a contradiction between the two bounds.
	
	Let
	\[\widetilde{F}(X,Y)= (X^{q^{i}} + \varphi(X)+bX^{q^k})Y-X(Y^{q^{i}}+\varphi(Y)+b Y^{q^k}),\]
	and let
	\[\widetilde{G}(X,Y,T)=(X^{q^{i}}T^{q^k-q^i} + \cdots +bX^{q^k})Y-X( Y^{q^{i}}T^{q^k-q^i}+\cdots+b Y^{q^k})\]
	be the homogenized polynomial of $\widetilde{F}(X,Y)$ with respect to $T$. The ideal points of $\cC$ are determined by
	\begin{equation}\label{eq:th1_G}
	\frac{\widetilde{G}(X,Y,0)}{X^{q}Y-XY^{q}}= b\frac{(X^{q^k}Y-Y^{q^k}X)}{X^{q}Y-XY^{q}}=b\prod _{\rho \in \F_{q^k}\setminus \F_{q}} (Y-\rho X).
	\end{equation}	
	Recall that $f(X)=X^{q^{i}} + \varphi(X)+bX^{q^k}$. By calculation, we have
	\[\frac{\partial \widetilde{F}(X,Y)}{\partial X}=f(X), \qquad \frac{\partial \widetilde{F}(X,Y)}{\partial Y}=f(Y).\]
	
	The curve $\cC$ has no singular points at infinity since \eqref{eq:th1_G} has pairwise distinct roots. Therefore the singular points of $\cC$ belong to those of the curve $\widetilde{\cC}$ determined by $\widetilde{F}(X,Y)$, namely the set
	\[\mathcal{S} =\{ (\xi, \eta) : \xi^{q^{i}} + \varphi(\xi)+b\xi^{q^k}=\eta^{q^{i}} + \varphi(\eta)+b\eta^{q^k}=0\}.\]
	
	Assume that $\widetilde{\cC}$ splits into two components $\widetilde{\cX}$ and $\widetilde{\cY}$ sharing no common irreducible component. As $\cC$ is contained in $\widetilde{\cC}$, if we can prove an upper bound on $I((\xi, \eta), \widetilde{\cX} \cap \widetilde{\cY})$ for arbitrary $\widetilde{\cX}$ and $\widetilde{\cY}$, it is also an upper bound on $I((\xi, \eta), \cX \cap \cY)$ for any two relatively prime components $\cX$ and $\cY$ of $\cC$. According to whether $\xi$ and $\eta$ equal $0$ or not, we consider the intersection number $I((\xi, \eta), \widetilde{\cX} \cap \widetilde{\cY})$ in the following three cases.
	\begin{itemize}
		\item Let $(\xi, \eta) \in \mathcal{S}$, with $\xi\eta\neq 0$. Then 
		\begin{align*}
		\widetilde{F}(X+\xi,Y+\eta)=&(X^{q^{i}} + \varphi(X)+bX^{q^k})(Y+\eta)-(X+\xi)( Y^{q^{i}}+\varphi(Y)+b Y^{q^k})\\
		=&(\eta X^{q^{i}} - \xi Y^{q^{i}}) + (Y X^{q^{i}}-XY^{q^{i}})+\cdots\\
		=& ( \overline{\eta} X - \overline{\xi}Y)^{q^i}+ (Y X^{q^{i}}-XY^{q^{i}})+\cdots,
		\end{align*}
		where $\overline{\eta}^{q^i}=\eta$, $\overline{\xi}^{q^i}=\xi$.

		If  $(\overline{\eta} X - \overline{\xi}Y)$ is not a factor 
		of $(Y X^{q^{i}}-XY^{q^{i}})$, then the intersection number 
		\[I((\xi, \eta), \widetilde{\cX} \cap \widetilde{\cY})=0\] 
		by Lemma \ref{le:intersection_number_m_m1_coprime}. 
		
		On the other hand, if  $(\overline{\eta} X - \overline{\xi}Y)$ is  a factor of $(Y X^{q^{i}}-XY^{q^{i}})$, then it is a simple factor and by Lemma \ref{le:intersection_number_linear_term}
			\[I((\xi, \eta), \widetilde{\cX} \cap \widetilde{\cY})=0 \text{ or }q^i.\]
		We want to determine when $(\overline{\eta} X - \overline{\xi}Y) \mid (Y X^{q^{i}}-XY^{q^{i}})$: this is equivalent to require $\overline {\eta}=U\overline{\xi}$, where $U^{q^i-1}=1$.	As $f(x)=0$ has {at most $q^{\ell-i}$} distinct roots, there are at most {$(q^{\ell-i}-1)(q^i-1)$} different pairs of $(\xi,\eta)$ such that $\xi\eta\neq 0$ and the intersection number $I((\xi, \eta), \widetilde{\cX} \cap \widetilde{\cY})=q^i$. 		
		\item Let $(\xi, \eta) \in \mathcal{S}\setminus \{(0,0)\}$, with either $\xi=0$ or $\eta=0$. By calculation, we have
		\begin{align*}
		\widetilde{F}(X,Y+\eta)&=(X^{q^{i}} + \varphi(X)+bX^{q^k})(Y-\eta)-X(a Y^{q^{i}}+\varphi(Y)+b Y^{q^k})\\
		&=(\eta X^{q^{i}}) + (Y X^{q^{i}}-XY^{q^{i}})+\cdots\\
		&= (\overline{\eta} X)^{q^i}+ X(Y X^{q^{i}-1}-Y^{q^{i}})+\cdots,
		\end{align*}		
		and
		\begin{align*}
		\widetilde{F}(X + \xi,Y)&=(X^{q^{i}} + \varphi(X)+bX^{q^k})(Y)-(X-\xi)( Y^{q^{i}}+\varphi(Y)+b Y^{q^k})\\
		&=(-\xi Y^{q^{i}}) + (Y X^{q^{i}}-XY^{q^{i}})+\cdots\\
		&=(-\overline{\xi} Y)^{q^i}+ Y(X^{q^{i}}-XY^{q^{i}-1})+\cdots.
		\end{align*}
		Thus, for all such singular points, the intersection number $I((\xi, \eta), \widetilde{\cX} \cap \widetilde{\cY})=0$ or $q^i$ by Lemma \ref{le:intersection_number_linear_term}. Again for $f(X)=0$, it has at most {$q^{\ell-i}-1$} distinct nonzero roots whence at most {$2(q^{\ell-i}-1)$} pairs of $(\xi,\eta)$.
		
		\item Finally the point $(0,0)$ is an ordinary singular point of multiplicity $q^i-q$ of $\cC$, i.e.\ the tangents at $(0,0)$ are all distinct. Hence, by Lemma \ref{le:ordinary_singular},
		\[I((0,0), \widetilde{\cX} \cap \widetilde{\cY}) \le \frac{(q^i-q)^2}{4}.\]
	\end{itemize}
	
	Summing up, the total intersection number of $\cX$ and $\cY$ has the following upper bound
	\begin{align}
		\label{eq:th1_upper}	
		\sum_{P \in \cX \cap\cY}\cI(P, \cX \cap\cY) &\le (q^{\ell-i}-1)(q^i-1)\times q^{i}+2(q^{\ell-i}-1)\times q^i + \frac{(q^i-q)^2}{4} \\
		\nonumber
		&= q^{\ell+i}+q^{\ell}-q^{2i}-q^i +\frac{(q^{i}-q)^2}{4}.
	\end{align}
	
	Assume that
	\[F(X,Y)=W_1(X,Y)W_{2}(X,Y)\cdots W_{r}(X,Y)\]
	is the decomposition of $F(X,Y)$ over $\F_{q^n}$ with $\deg W_i=d_i$ and $\sum_{i=1}^{r} d_i=q^k-q^i$. As we have shown that for any two components $\cX$ and $\cY$, their total intersection number has an upper bound,  $W_i$ and $W_j$ must be relatively prime for any distinct $i$ and $j$.
	
	By Lemma  \ref{le:splitting_of_irreducible_polys}, there exist natural numbers $s_i$ such that $W_{i}$ splits into $s_i$ absolutely irreducible factors over $\overline{\F}_{q^n}$, each of degree $d_i/s_i$. As $F(X,Y)$ is assumed without absolutely irreducible component over $\F_{q^n}$, we have $s_i>1$ for $i=1,2,\dots, r$. Define two polynomials $A(X,Y)$ and $B(X,Y)$ by
	\[A(X,Y)=\prod _{j=1}^{\lfloor s_i/2 \rfloor }  Z_i^j (X,Y), \qquad B(X,Y)=\prod _{j=\lceil s_i/2 \rceil}^{s_i }  Z_i^j (X,Y),\]
	where $Z_i^{1}(X,Y), \ldots, Z_{i}^{s_i}(X,Y)$ are the absolutely irreducible components of $W_{i}(X,Y)$. Let $\alpha$ and $\alpha+\beta$ be the degrees of $A(X,Y)$ and $B(X,Y)$ respectively. Then 
	$$2\alpha+\beta=q^k-q, \qquad \beta\leq \alpha, \qquad \beta \leq \frac{q^k-q}{3}.$$
	Let $\mathcal{A}$ and $\mathcal{B}$ be the curves defined by $A(X,Y)$ and $B(X,Y)$, respectively. It is clear that
	\[	(\deg A)( \deg B)=(\alpha+\beta)\alpha = \frac{(q^k-q)^2-\beta^2}{4} \ge \frac{2}{9}(q^k-q)^2.\]
	
	By B\'ezout's Theorem (see Theorem \ref{th:bezout})
	\begin{equation}\label{eq:th1_lower}
	\sum_{P \in \mathcal{A}\cap \mathcal{B}}\cI(P, \mathcal{A} \cap \mathcal{B})=(\deg A)( \deg B) \ge \frac{2}{9}(q^k-q)^2.
	\end{equation}
	
	On the other hand, we have already obtained an upper bound on the total intersection number of any two coprime components of $\cC$ in \eqref{eq:th1_upper}. Together with \eqref{eq:th1_lower}, we get
	\[q^{\ell+i}+q^\ell-q^{2i}-q^i +\frac{(q^{i}-q)^2}{4}\ge \frac{2}{9}(q^{k}-q)^2,\]
	which contradicts the condition. Therefore we finish the proof.	
\end{proof}

\begin{remark}\label{remark:main_0_conditions}
	It is clear that $\ell\le k$ in Theorem \ref{th:main_0}. It is not difficult to check that  
	the condition on $\ell$, $i$ and $q$ in Theorem \ref{th:main_0} is always satisfied when $q>5$ and  $k>i$. For small $q$ and $i<k$, this condition is not satisfied exactly when $q$, $k$ and $i$ satisfy one of the following conditions: 
	\begin{itemize}
		\item $q=2$, $k< 7$ and $i=k-1$; 
		\item $q=2$, $k\geq7$ and $i=k-2,k-1$;
		\item $q=3$ and $i=k-1$;
		\item $q=4$ and $(k,i)=(2,1), (3,2)$;
		\item $q=5$ and $(k,i)=(2,1)$.
	\end{itemize}
	If we further assume that $f$ satisfies $\ker(f)\le q$ which is necessary under the condition that $f$ is scattered, then $\ell-i\le 1$. Under this extra assumption, the condition in Theorem \ref{th:main_0} is satisfied exactly in one of the following cases: 
	\begin{itemize}
		\item  $q=2,3$ and $k-i\ge 2$;
		\item $q=4$, $k>i$ and $(k,i)\neq  (2,1), (3,2)$;
		\item $q=5$, $k>i$ and $(k,i)\neq (2,1)$;
		\item $q>5$ and  $k>i$.
	\end{itemize}
\end{remark}

\begin{theorem}\label{th:main_0_function}
	Let $\cC$ be the curve defined by $F(X,Y)=0$, where
	$$F(X,Y)=(X^{q^{i}} + \varphi(X)+bX^{q^k})Y-X(Y^{q^{i}}+\varphi(Y)+b Y^{q^k}).$$
	For $f$ defined by \eqref{eq:f_main_0}, we assume $\ker(f)= q^{\ell-i}$.
	If one of the following sets of conditions is satisfied
	\begin{itemize}
		\item $\ell-i>1$;
		\item $\gcd(k,n)>1$ and $k\le \frac{n}{4}$;
		\item $q^{\ell+i}+q^\ell-q^{2i}-q^i +\frac{(q^{i}-q)^2}{4}< \frac{2}{9}(q^{k}-q)^2$ and $k\le \frac{n}{4}$;
	\end{itemize} 
	then $\cC$ contains at least one point $(x,y)$ such that $\frac{x}{y}\notin\F_{q}$ and the function defined by \eqref{eq:f_main_0} is not a scattered polynomial.
\end{theorem}
\begin{proof}
	When $\ell-i>1$, $f$ has more than $q$ roots. It follows that there are $x$ and $y\in \F_{q^n}$ such that $f(x)=f(y)=0$ and $y/x\not \in \F_q$. Hence, by definition, $f$ is not scattered.
	
	When $\gcd(k,n)>1$, from \eqref{eq:th1_G} we derive that there exist simple points $(1,\xi,0)$ at infinity, with $ \xi \in \F_{q^n}\setminus \F_{q}$ and therefore there exists at least one absolutely irreducible $\mathbb{F}_{q^n}$-rational component through one of them.	
	
	When $\gcd(k,n)=1$ and the second condition on $k$, $q$ and $i$ holds, by Theorem \ref{th:main_0}, $\cC$ has one $\mathbb{F}_{q^n}$-rational absolutely irreducible component. 
	
	By the Hasse-Weil Theorem (see Theorem \ref{th:HW-bound}), the number of $\F_{q^n}$-rational points of $\cC$ satisfies
	\[\#\cC(\F_{q^n})\geq q^n+1-(q^k-q-1)(q^k-q-2)\sqrt{q^n}.\]
	
	In the proof of Theorem \ref{th:main_0}, we know that $\cC$ has exactly $q^{k}-q$ points at infinity; see \eqref{eq:th1_G}. Plugging $Y=U X$ into $F$, we get
	\begin{align*}
	F(X,UX)&=\frac{(X^{q^{i}} + \varphi(X)+bX^{q^k})XU -X(U^{q^{i}} X^{q^{i}}+\cdots +b U^{q^k} X^{q^k})}{X^{q+1}(U-U^q)}\\
	&=\frac{(U-U^{q^i}) X^{q^{i}+1} + \cdots}{X^{q+1}(U-U^q)}\\
	&=X^{q^i-q} \frac{U-U^{q^i}}{U-U^q} + \cdots.
	\end{align*}
	Note that the points $(x,y)$ of $\cC$ satisfying $\frac{x}{y} \in \F_q$ belong to lines $Y=\delta X$ for certain $\delta \in \F_q$. By the previous equation on $F(X,UX)$ we see that $F(X,\delta X)$ is not the zero polynomial. Thus, there are at most $q(q^k-q)$ of such points, because $\deg F=q^k-q$. Therefore the existence of a suitable point $(x, y)$ satisfying $\frac{x}{y}\notin \F_{q}$ is ensured whenever 
	\[q^n+1-(q^k-q-1)(q^k-q-2)\sqrt{q^n}>q^{k}-q+q(q^k-q)=(q^k-q)(q+1),\]
	which is implied by our assumptions that $k\le \frac{n}{4}$.
\end{proof}

From Lemma \ref{le:link}, Theorem \ref{th:main_0_function} and Remark \ref{remark:main_0_conditions}, we get the following result.
\begin{corollary}\label{coro:main_0_clear}
	When $q>5$, the unique exceptional scattered monic polynomial of index $0$ is $X^{q^k}$.
\end{corollary}

\subsection{Exceptional scattered polynomials of index $1$}
Next, we concentrate on the case $t=1$. Let 
\begin{equation}\label{eq:f_main_1}
	f(X)=X + \varphi(X) + \lambda X^{q^k},	
\end{equation}
where $\varphi(X)=\sum _{j=1}^{k-1} \alpha_j X^{q^{j}} \in \F_{q^n}[X]$ and $\lambda \neq 0$.

Different from the proof of Theorem \ref{th:main_0}, we cannot apply Lemma \ref{le:intersection_number_m_m1_coprime} and Lemma \ref{le:intersection_number_linear_term} here to estimate the intersection numbers. To overcome this difficulty, we need to investigate  the branches centered at the singular points of the associated curve and apply local quadratic transformations to show the following result.

\begin{theorem}\label{th:main}
	Let $\lambda \in \F_{q^n}^*$ and $\varphi(X)=\sum _{j=1}^{k-1} \alpha_j X^{q^{j}}$. Let $\cC$ be the curve defined by $F(X,Y)=0$, where
	\begin{equation}\label{F(X,Y)}
	F(X,Y)=\frac{X^qY-XY^q+X^q\varphi(Y)-\varphi(X)Y^q+\lambda(X^qY^{q^k}-X^{q^k}Y^q)}{X^qY-XY^q}.
	\end{equation}
	If $k\geq3$, then $\cC$ has an absolutely irreducible $\F_{q^n}$-rational component. 
\end{theorem}
\begin{proof}
	Suppose, by way of contradiction, that $\cC$ has no absolutely irreducible components over $\F_{q^n}$. We divide our proof into four steps: first we find all the singular points of $\cC$ and determine the branches centered at them. Second we assume that $\cC$ splits into two components $\cX$ and $\cY$ sharing no common irreducible component, then we prove an upper bound on the total intersection number of $\cX$ and $\cY$. Then, under the assumption that $\cC$ has no absolutely irreducible components over $\F_{q^n}$, we decompose $F(X,Y) = A(X,Y) B(X,Y)$ and obtain a lower bound on $(\deg A)( \deg B)$. Finally, by using B\'ezout's theorem (see Theorem \ref{th:bezout}), we get a contradiction between the two bounds.
	
	\vspace*{2mm}
	\noindent\textbf{Step 1:}
	To determine the singular points of $\cC$, let us consider 
	\[\widetilde {F}(X,Y)=X^qY-XY^q+X^q\varphi(Y)-\varphi(X)Y^q+\lambda(X^qY^{q^k}-X^{q^k}Y^q)=(X^qY-XY^q)F(X,Y)\] 
	and let $\widetilde {\cC}$ denote the curve defined by $\widetilde {F}(X,Y)=0$.
	Also, let 
	\begin{align*}
		\widetilde {G}(X,Y,T)=&(X^qY-XY^q)T^{q^k-1}+ X^q\left(\sum _{j=1}^{k-1} \alpha_j Y^{q^{j}}T^{q^k-q^j}\right)\\
		 &-  Y^q\left(\sum _{j=1}^{k-1} \alpha_j X^{q^{j}}T^{q^k-q^j}\right) + \lambda(X^qY^{q^k}-X^{q^k}Y^q)
	\end{align*}
	be the homogenized polynomial of $\widetilde {F}(X,Y)$ with respect to $T$. 
	
	As
	\[ \tilde{F}(X+a, Y+b) = a^qY -Xb^q +X^qb - aY^q + \cdots,\]
	the curve $\widetilde{\cC}$ has only one affine singular point $(0,0)$ which does not belong to $\cC$. 
	
	By considering the zeros of 
	\[\widetilde {G}(X,1,0)= \lambda ( X^q-X^{q^k})\]
	and 
	\[\widetilde {G}(1,Y,0)=\lambda ( Y^{q^k}-Y^q),\]
	we see the points at infinity of $\cC$ are $P=(1,0,0)$, $Q=(0,1,0)$, $R_{\xi}=(1,\xi,0)$, $\xi \in \F_{q^{k-1}}\setminus \F_q$, and $S_{\xi}=(1,\xi,0)$, $\xi \in  \F_q^*$. The multiplicities of $P$, $Q$ and $S_{\xi}$ are $q-1$, and the multiplicities of $R_{\xi}$ are $q$.
	
	Now we determine the branches centered at these points contained in the curve $\cC$.
	
	\begin{itemize}
		\item In order to study the point $P=(1,0,0)$ we consider the polynomial
		\begin{align}
		\nonumber	
		&H(X,Y)=\widetilde {G}(1,Y,X)/Y= X^{q^k-1}-Y^{q-1}X^{q^k-1}\\
		\label{H(X,Y)}	& +  \sum_{j=1}^{k-1}\alpha_j Y^{q^j-1}X^{q^k-q^j}
		- Y^{q-1} \sum_{j=1}^{k-1}\alpha_jX^{q^k-q^j}+\lambda Y^{q^k-1}- \lambda Y^{q-1}.
		\end{align}
		The branches centered at $(0,0)$ belonging to the curve $\cH$ defined by $H(X,Y)=0$ correspond to the branches centered at $P$ and contained in  the curve $\cC$. We can immediately see that the origin is not an ordinary $(q-1)$-point of $\cH$. 
		
		As the tangent line at the origin is not the line $X=0$, we can apply the local quadratic transformation to $\cH$ and obtain a curve given by $H^{\prime}(X,Y)=H(X,XY)/X^{q-1}$. By calculation,
		\begin{align*}
			H^{\prime}(X,Y)=&X^{q^k-q}-Y^{q-1}X^{q^k-1}+
			\sum_{j=1}^{k-1}\alpha_j Y^{q^j-1}X^{q^k-q}\\
			&- Y^{q-1}\sum_{j=1}^{k-1}\alpha_jX^{q^k-q^j}+\lambda Y^{q^k-1}X^{q^k-q} -\lambda Y^{q-1}.
		\end{align*}
		Since the origin is still not an ordinary $(q-1)$-point of this new curve defined by $H'$ and the line $X=0$ is not the tangent line at the origin, we can apply again the local quadratic transformation. After applying it $N-1=(q^k-1)/(q-1)-1$ times, we obtain
		\begin{align*}
		H^{(N-1)}(X,Y)=&X^{q-1}-Y^{q-1}X^{q^k-1}
		+\sum_{j=1}^{k-1}\alpha_j Y^{q^j-1}X^{q^k-q + (N-1)(q^j-q)}\\
		&- Y^{q-1}\sum_{j=1}^{k-1}\alpha_j X^{q^k-q^j}
		+\lambda Y^{q^k-1}X^{(N-1)(q^k-q)} -\lambda Y^{q-1}.
		\end{align*}
		
		One more time of local quadratic transformation leads to
		\[1-Y^{q-1}X^{q^k-1}
		+\sum_{j=1}^{k-1}\alpha_j Y^{q^j-1}X^{q^k-q + N(q^j-q)}
		- Y^{q-1}\sum_{j=1}^{k-1} \alpha_jX^{q^k-q^j}
		+\lambda Y^{q^k-1}X^{N(q^k-q)} -\lambda Y^{q-1},\]
		which is denoted by $H^{(N)}(X,Y)$.
		We use $\cH^{(N)}$ to denote the curve defined by $H^{(N)}$. By Theorem \ref{th:geometrc_transform}, the branches of $\cH$ centered at the origin $(0,0)$ are mapped bijectively to the branches of $\cH^{(N)}$  centered at affine points on the line $X=0$, which are exactly the $q-1$ distinct simple points $(0,\eta_i)$, $i=1,\ldots,q-1$, where $\eta_i^{q-1}=1/\lambda$. 
		
		By Lemma \ref{le:branch_(t,t+..)}, the branch of $\cH^{(N)}$ centered at $(0, \eta_i)$ has a representation $(t, \eta_i + \sum^\infty_{j=1}c^{(i)}_jt^j )$ for certain $c^{(i)}_1, c^{(i)}_2, \ldots \in \F_{q^n}$. This means that the branches centered of $\cH$ at the origin $(0,0)$ have the representations
		$$\left(t, \eta_{i} t^N+ t^N\sum^\infty_{j=1}c^{(i)}_jt^j \right), \qquad i=1,\ldots,q-1.$$
		
		Using \eqref{H(X,Y)} it is possible to see that if $(t, \eta_{i} t^N+ct^{N+\beta}+\cdots)$ defines a branch of $\cH$, then plugging it into $YH(X,Y)$ we have
		\begin{align*}
		&(\eta_{i} t^N+ct^{N+\beta}+\cdots)H(t, \eta_{i} t^N+ct^{N+\beta}+\cdots)\\
		=&(\eta_{i} t^N+ct^{N+\beta}+\cdots)t^{N(q-1)}-(\eta_{i}^q t^{Nq}+c^qt^{Nq+\alpha q}+\cdots)t^{N(q-1)}\\
		&+\sum_{j=1}^{k-1}\alpha_j t^{q^k-q^j}\left((\eta_i^{q^j}t^{Nq^j} + c^{q^j} t^{Nq^j + \beta q^j} + \cdots)- (\eta_i^{q}t^{Nq} + c^{q} t^{Nq + \beta q} + \cdots) \right)\\
		&+\lambda (\eta_{i}^{q^k} t^{Nq^k}+c^{q^k}t^{Nq^k+\beta q^k}+\cdots)- \lambda (\eta_{i}^q t^{Nq}+c^qt^{Nq+\beta q}+\cdots)\\
		= &\eta_{i} t^{Nq}+ct^{Nq+\beta}+\cdots-\eta_{i}^q t^{2Nq-N}-c^qt^{2Nq-N +\beta q}+\cdots\\
		&+\sum_{j=1}^{k-1}(\alpha_j \eta_i^{q^j}t^{q^k+(N-1)q^j} + \alpha_j c^{q^j} t^{q^k+(N+\beta-1)q^j} + \cdots)\\
		&-\sum_{j=1}^{k-1} (\alpha_j \eta_i^{q}t^{q^k-q^j+Nq} + \alpha_j c^{q} t^{q^k-q^j +Nq + \beta q} + \cdots)\\
		&+\lambda \eta_{i}^{q^k} t^{Nq^k}+\lambda c^{q^k}t^{Nq^k+\beta q^k}+\cdots- \lambda \eta_{i}^q t^{Nq}-\lambda c^qt^{Nq+\beta q}-\cdots
		\end{align*}
		which should equals $0$. As the coefficients of $t^j$ must be $0$ for all $j$, we concentrate on the coefficient of $t^{Nq+\beta}$. 
		
		Set $\alpha_0=1$ and let $\overline{j}=\max\{j:\alpha_j\neq 0\}$. By calculation, we see that $\beta= q^k-q^{k-\overline{j}}$ and $c=\alpha_{\overline{j}}\eta_i^q$ which implies that
		$$(t, \eta_{i} t^N+\alpha_{\overline{j}}\eta_i^q t^{N+q^k-q^{k-\overline{j}}}+\cdots), \qquad i=1,\ldots,q-1$$
		represent the branches of $\cH$ centered at the origin $(0,0)$.

		\item The branches centered at $Q=(0,1,0)$ are exactly those represented by
		$$(\eta_{i} t^N+\alpha_{\overline{j}}\eta_i^q t^{N+q^k-q^{k-\overline{j}}}+\cdots,t), \qquad i=1,\ldots,q-1,$$
		since $F(X,Y)$ is symmetric for $X$ and $Y$.
		
		\item The branches centered at $S_{\xi}=(1,\xi,0)$, $\xi \in \F_{q}^*$ are exactly of type
		$$(t, \xi  t+ \eta_{i} t^N+\alpha_{\overline{j}}\eta_i^q t^{N+q^k-q^{k-\overline{j}}}+\cdots), \qquad i=1,\ldots,q-1,$$
		since the curve $\cC$ is fixed by $Y \mapsto Y+\xi X $.
		
		\item Finally we consider the points $R_{\xi}=(1,\xi,0)$, $\xi \in  \F_{q^{k-1}} \setminus \F_{q}$. Let us consider
		\begin{align*}
		L(X,Y)=&\widetilde {G}(1,Y+\xi,X)\\
		=&X^{q^k-1}Y+X^{q^k-1}(\xi-\xi^q) -Y^{q}X^{q^k-1}\\
		&+\sum_{j=1}^{k-1}\alpha_j (Y^{q^j}+\xi^{q^j}) X^{q^k-q^j} -(Y^q+\xi^q)\sum_{j=1}^{k-1}\alpha_j X^{q^k-q^j} \\
		&+\lambda Y^{q^k}- \lambda Y^{q}.
		\end{align*}
		As $-\lambda Y^q$ is the term of lowest degree in $L(X,Y)$, we see that $R_{\xi}=(1,\xi,0)$ are non-ordinary singular points of $\cC$ of multiplicity $q$. After the local quadratic transformation we obtain
		\begin{align*}
		L'(X,Y)=&L(X,XY)/X^q\\
		=&X^{q^k-q}Y+X^{q^k-q-1}(\xi-\xi^q) -Y^{q}X^{q^k-1}\\
		&+\sum_{j=1}^{k-1}\alpha_j Y^{q^j}  X^{q^k-q} -Y^q\sum_{j=1}^{k-1}\alpha_j X^{q^k-q^j} \\  
		&+\sum_{j=1}^{k-1}\alpha_j(\xi^{q^j}-\xi^q) X^{q^k-q^j-q}+\lambda Y^{q^k}X^{q^k-q}- \lambda Y^{q}.
		\end{align*}
		Let $\alpha_0=1$ and $j_0=\max\{j : \alpha_j(\xi^{q^j}-\xi^q)\neq 0 \}$. Now we have to compute the branches in two different cases.
		
		\textbf{Case 1. $j_0=0$:} By applying the local quadratic transformation to $L$ for $N=(q^k-q)/q=q^{k-1}-1$ times, we obtain
		\begin{align*}L^{(N)}(X,Y)=&X^{q^{k-1}+q-2}Y+X^{q-1}(\xi-\xi^q) +\sum_{j=2}^{k-1}\alpha_jY^{q^j}X^{q^k-q^j+N(q^j-q)}\\
		&-\sum_{j=2}^{k-1}\alpha_jY^{q}X^{q^k-q^j}-Y^{q}X^{q^k-1}+\lambda Y^{q^k}X^{(q^k-q)^2/q}- \lambda Y^{q}.\end{align*}
		Since now the tangent line at $(0,0)$ is $X=0$, we use the following quadratic transformation
		\begin{align*}
		\overline{L}(X,Y)=&L^{(q^{k-1}-1)}(XY,Y)/Y^{q-1}\\
		=&X^{q^{k-1}+q-2}Y^{q^{k-1}}+X^{q-1}(\xi-\xi^q) -Y^{q^k}X^{q^k-1}+\lambda Y^{q^{2k-1}-q^k+1}X^{(q^k-q)^2/q}- \lambda Y.
		\end{align*}
		Now $-\lambda Y$ is the term of lowest degree in $\overline{L}$, whence the origin is a simple point of it and therefore there exists a unique branch of $\cC$ centered at $R_{\xi}=(1,\xi,0)$,  $\xi \in  \F_{q^{k-1}} \setminus \F_{q}$.
		
		\textbf{Case 2. $j_0\neq 0$:} by applying the local quadratic transformation to $L$ for $N=(q^k-q^{j_0})/q-1=q^{k-1}-q^{j_0-1}-1$ times, we obtain
		\begin{align*}
		L^{(N)}(X,Y)=&X^{q^k-1-N(q-1)}Y+X^{q^k-1-Nq}(\xi-\xi^q) -Y^{q}X^{q^k-1}\\
		&+\sum_{j=1}^{j_0}\alpha_j Y^{q^j}  X^{q^k-q^j + N(q^j-q)} -Y^q\sum_{j=1}^{j_0}\alpha_j X^{q^k-q^j} \\  
		&+\sum_{j=1}^{j_0}\alpha_j(\xi^{q^j}-\xi^q) X^{q^k-q^j-Nq}+\lambda Y^{q^k}X^{N(q^k-q)}- \lambda Y^{q}.
		\end{align*}
		Now the terms of lowest degree is $\alpha_{j_0}(\xi^{q^{j_0}}-\xi^q)X^q-\lambda Y^q$. Let $Y_1 := \overline{\alpha}_{j_0}(\xi^{q^{j_0-1}}-\xi)X- \overline{\lambda} Y$, where $\overline{\alpha}_{j_0}^q=\alpha_{j_0}$ and $\overline{\lambda}^q=\lambda$. Rewrite $L^{(N)}(X,Y)$ as a polynomial $L_1(X, Y_1)$. In $L_1(X, Y_1)$, the possible terms are
		\begin{align*}
		&X^{q^k-N(q-1)}, \quad X^{q^k-1-N(q-1)}Y_1, \quad X^{q^k-1-Nq}, \quad X^{q^k+q-1}, \quad Y_1^{q}X^{q^k-1},\\
		&X^{q^k+ N(q^j-q)}, \quad Y_1^{q^j}X^{q^k-q^j + N(q^j-q)}, \quad X^{q^k-q^j+q}, \quad Y_1^qX^{q^k-q^j}, \quad X^{q^k-q^j-Nq},\\
		&X^{N(q^k-q)+q^k}, \quad Y_1^{q^k}X^{N(q^k-q)}, \quad Y_1^q,
		\end{align*}
		where $1\le j\leq j_0$. It is important to notice that 
		\begin{itemize}
			\item the coefficients of $X^{q^k-1-Nq}$ is still $\xi-\xi^q\neq 0$;
			\item the coefficient of $Y_1^q$ is nonzero;
			\item the terms of $X^j$ listed in increasing order of their exponents are
			\[ X^{q^k-q^{j_0-1}-Nq}, X^{q^k-q^{j_0-2}-Nq},\ldots, X^{q^k-q-Nq}, X^{q^k-1-Nq},\ldots .\]
		\end{itemize}
		Then we apply the local quadratic transformation on $L_1(X, Y_1)$ by taking $L_1'(X,XY_1)=L_1(X,XY_1)/X^q$. Again, let $j^{(1)}_0$ be the maximum integer $j$ such that the coefficient of $X^{q^k-q^j-Nq}$ is nonzero for $0\le j< j_0$. Clearly $j^{(1)}_0<j_0$. 
		
		If $j^{(1)}_0=0$, then we follow the steps in case 1 and prove the same results; otherwise after applying the same local quadratic transformation enough number of times (say $N_1$ times), we get $L_1^{(N_1)}(X, Y_1)$ in which $X^q$ and $Y_1^q$ are the terms of lowest degree again. Let us denote the sum of these two terms by $(u_1 X + v_1Y_1)^q$ for some $u_1,v_1\in \F_{q^n}^*$. Then we set $Y_2= u_1X + v_1Y_1$ and rewrite $L_1^{(N_1)}(X, Y_1)$ as a polynomial $L_2(X, Y_2)$. In $L_2$, we can still guarantee that
		\begin{itemize}
			\item the coefficients of $X^{q^k-1-(N+N_1)q}$ is still $\xi-\xi^q\neq 0$;
			\item the coefficient of $Y_2^q$ is nonzero;
			\item the terms of $X^j$ listed in increasing order of their exponents are
			\[ X^{q^k-q^{j^{(1)}_0-1}-(N+N_1)q}, X^{q^k-q^{j^{(1)}_0-2}-(N+N_1)q},\ldots, X^{q^k-q-(N+N_1)q}, X^{q^k-1-(N+N_1)q},\ldots .\]
		\end{itemize}
		Then we take $j^{(2)}_0$ be the maximum integer $j$ such that the coefficient of $X^{q^k-q^j-(N+N_1)q}$ is nonzero for $0\le j<j^{(1)}_0$. Again we consider two cases depending on whether $j^{(2)}_0=0$ or not. By doing all this process recursively, after finite steps (say $M$ steps), we have $j^{(M)}_0=0$ and end up with the calculation as in Case 1 and there exists a unique branch of $\cC$ centered at $R_{\xi}=(1,\xi,0)$,  $\xi \in  \F_{q^{k-1}} \setminus \F_{q}$.
	\end{itemize}
	
	\noindent\textbf{Step 2:} Suppose now that the curve $\cC$ splits into two components $\cX$ and $\cY$ sharing no common irreducible component. It follows that $\cX$ and $\cY$ have no branches in common. Their intersection points are clearly singular points for $\cC$. First of all, by Lemma \ref{le:2_irre_curve_no_branch} we observe that the multiplicity of intersection of $\cX$ and $\cY$ in one of the points $R_{\xi}=(1,\xi,0)$ is $I(R_{\xi}, \cX \cap \cY) =0$ since there is a unique branch of $\cC$ centered at each of these points.
	
	Next we consider the polynomial $H(X,Y)$ given by \eqref{H(X,Y)}. Let $U(X,Y)$ and $V(X,Y)$ be the two polynomials defining the components $\cX$ and $\cY$ such that $U$ and $V$ have no common factors. Then 
	\[U(X,Y)=Y^m+U_0(X,Y), \qquad V(X,Y)=Y^{q-1-m}+V_0(X,Y),\]
	for some $0\leq m\leq q-1$ and $U_0(X,Y),V_0(X,Y)\in \F_{q^n}[X,Y]$ satisfying $\deg U_0>m$ and $\deg V_0>q-1-m$. 
	
	Now let us look at the branches of $\cH$ centered at the point $P=(1,0,0)$. If a branch represented by $\gamma_{i}=(t, \eta_{i} t^N+\alpha_{\overline{j}}\eta_i^q t^{N+q^k-q^{k-\overline{j}}}+\cdots)$, $i=1,\ldots,q-1$ belongs to $\cX$, then in particular the coefficient of $t^{mN}$ in $U(\gamma_i)$ must vanish. This coefficient is given by 
	$$\eta_{i}^m+\sum_{j=0}^{m-1}\alpha_j \eta^{j},$$
	where $\alpha_j$ is the coefficient of $X^{N(m-j)}Y^{j}$ in $U(X,Y)$. Analogously, $\gamma_{i}\in \cY$ is equivalent to \[\eta_{i}^{q-1-m}+\sum_{j=0}^{q-m-2}\beta_j \eta^{j}=0,\]
	where $\beta_j$ is the coefficient of $X^{N(q-1-m-j)}Y^{j}$ in $V(X,Y)$. Therefore the number of branches contained in $\cX$ and $\cY$ is at most $m$ and $q-1-m$ respectively. On the other hand the total number of branches is $q-1$ whence
	$$\gamma_{i} \in \cX \iff \eta_{i}^m+\sum_{j=0}^{m-1}\alpha_j \eta^{j}=0, \qquad \gamma_{i} \in \cY \iff \eta_{i}^{q-1-m}+\sum_{j=0}^{q-m-2}\beta_j \eta^{j}=0.$$
	If $\gamma_i\notin \cX$ then $\cI(P,\cX\cap \gamma_i)\leq mN$, since the coefficient of $t^{mN}$ does not vanish; see Definition \ref{def:intersection_multi}.  Therefore, by Theorem \ref{th:intersection_branch}
	\[\cI(P,\cX\cap\cY)=\sum_{\gamma_{i} \in \cY} \cI(P,\cX\cap \gamma_i)\leq (q-1-m)mN\leq\frac{1}{4} (q-1)^2N=\frac{1}{4} (q-1)(q^k-1).\]
		
	For all the other points, $Q$ and $S_{\xi}$ the same argument holds. Therefore, for any two relatively prime components of $\cC$, the sum of their intersection number
	
	\begin{equation}\label{eq:sum_intersection_le}
	\sum_{P\in \cX\cap\cY}\cI(P,\cX\cap\cY)\le  \frac{1}{4} (q-1)(q^k-1)(q+1)=\frac{1}{4} (q^2-1)(q^k-1).
	\end{equation}
	
	\noindent\textbf{Step 3:} Assume that
	\[F(X,Y)=W_1(X,Y)W_{2}(X,Y)\cdots W_{r}(X,Y)\]
	is the decomposition of $F(X,Y)$ over $\F_{q^n}$ with $\deg W_i=d_i$ and $\sum_{i=1}^{r} d_i=q^k-1$. As there are $q^{k-1}+1$ singular points on $\cC$ for arbitrary $n$,  $W_i$ and $W_j$ must be relatively prime for any distinct $i$ and $j$.
	
	By Lemma  \ref{le:splitting_of_irreducible_polys}, there exist natural numbers $s_i$ such that $W_{i}$ splits into $s_i$ absolutely irreducible factors over $\overline{\F}_{q^n}$, each of degree $d_i/s_i$. Since we are assuming that $F(X,Y)$ does not have absolutely irreducible components defined over $\F_{q^n}$, we have $s_i>1$ for $i=1,2,\dots, r$. Define two polynomials $A(X,Y)$ and $B(X,Y)$ by
	\[A(X,Y)=\prod _{j=1}^{\lfloor s_i/2 \rfloor }  Z_i^j (X,Y), \qquad B(X,Y)=\prod _{j=\lceil s_i/2 \rceil}^{s_i }  Z_i^j (X,Y),\]
	where $Z_i^{1}(X,Y), \ldots, Z_{i}^{s_i}(X,Y)$ are the absolutely irreducible components of $W_{i}(X,Y)$. Let $\alpha$ and $\alpha+\beta$ be the degrees of $A(X,Y)$ and $B(X,Y)$ respectively. Then 
	$$2\alpha+\beta=q^k-1, \qquad \beta\leq \alpha, \qquad \beta \leq \frac{q^k-1}{3}.$$
	Let $\mathcal{A}$ and $\mathcal{B}$ be the curves defined by $A(X,Y)$ and $B(X,Y)$, respectively. It is clear that
	\[	(\deg A)( \deg B)=(\alpha+\beta)\alpha = \frac{(q^k-1)^2-\beta^2}{4} \ge \frac{2}{9}(q^k-1)^2.\]
	
	\noindent\textbf{Step 4:} By B\'ezout's Theorem (see Theorem \ref{th:bezout})
	\begin{equation}\label{eq:sum_intersection_ge}
	\sum_{P \in \mathcal{A}\cap \mathcal{B}}\cI(P, \mathcal{A} \cap \mathcal{B})=(\deg A)( \deg B) \ge \frac{2}{9}(q^k-1)^2.
	\end{equation}
	
	On the other hand, we have already obtained an upper bound on the total intersection number of any two coprime components of $\cC$ in \eqref{eq:sum_intersection_le}. Together with \eqref{eq:sum_intersection_ge}, we get
	
	\[ \frac{2}{9}(q^k-1)^2\le \sum_{P \in \mathcal{A}\cap \mathcal{B}}\cI(P, \mathcal{A} \cap \mathcal{B})\le \frac{1}{4} (q^2-1)(q^k-1),\]
	which never holds for $k\geq 3$. Hence we get a contradiction and finish the proof.
\end{proof}

\begin{theorem}\label{th:main_1_function}
	Let $\lambda \in \F_{q^n}^*$. Let $\cC$ be the curve defined as in \eqref{F(X,Y)}. If $3\leq k\leq n/4$,
	then $\cC$ has an  $\F_{q^n}$-rational point $(x,y)$ such that $\frac{x}{y} \notin \F_q$ and the function defined by \eqref{eq:f_main_1} is not a scattered polynomial.
\end{theorem}
\begin{proof}
	By Theorem \ref{th:main} we know that $\cC$ has an absolutely irreducible $\F_{q^n}$-component. By the Hasse-Weil Theorem (see Theorem \ref{th:HW-bound}) the number of $\F_{q^n}$-rational points of $\cC$ satisfies
	\[\#\cC(\F_{q^n})\geq q^n+1-(q^k-1)(q^k-2)\sqrt{q^n}.\]
	
	By plugging $Y=UX$ into $F(X,Y)$, we get
	\[F(X,UX)=1+\sum_{i=1}^{k-1}\alpha_i X^{q^i-1}\frac{U^{q^i}-U^q}{U-U^q} + \lambda X^{q^k-1} \frac{U^{q^k}-U^q}{U-U^q}.\]
	No matter which value $U$ takes, it can never be a zero polynomial. Hence a line cannot be a component of $\cC$.  Since the points $(x,y)$ of $\cC$ satisfying $\frac{x}{y} \in \F_q$ belong to lines $Y=\delta X$, $\delta \in \F_q$ and $\deg F=q^k-1$, there are at most $q(q^k-1)$ of such points.
	
	On the other hand, in Step 1 of the proof of Theorem \ref{th:main}, $\cC$ has exactly $q^{k-1}+1$ points at infinity. Therefore the existence of a suitable point $(x, y)$ satisfying $\frac{x}{y}\notin \F_{q}$ is ensured whenever 
	\[q^n+1-(q^k-1)(q^k-2)\sqrt{q^n}>q^{k-1}+1+q(q^k-1),\]
	which is implied by our assumptions on $k$ and $n$.
\end{proof}
From Lemma \ref{le:link}, and Theorem \ref{th:main_1_function}, we can prove the following result.
\begin{corollary}\label{coro:main_1_clear}
	The only exceptional scattered monic polynomials $f$ of index $1$ over $\F_{q^n}$ are $X$ and $bX + X^{q^2}$ where $b\in \F_{q^n}$ satisfying $\mathrm{Norm}_{q^n/q}(b)\neq 1$. In particular, when $q=2$, $f(X)$ must be $X$.
\end{corollary}
\begin{proof}
	First we show that $\mathrm{Norm}_{q^n/q}(b)\neq 1$ is a necessary condition for $bX + aX^{q} + X^{q^2}$ having at most $q$ zeros for every $a\in \F_{q^n}$.
	
	Clearly there exist polynomials of degree $q^2$ having exactly $q^2$ zeros in $\mathbb{F}_{q^n}$. Let ${u,v}$ be an $\F_q$-basis of an arbitrary $2$-dimensional $\F_q$-subspace $W$ of $\F_{q^n}$. Define a linearized polynomial $\overline{g}$ of degree $q^2$ by
	\[\overline{g}(X)=  \det \left(
	\begin{matrix}
	X & X^q & X^{q^2} \\ 
	u & u^q & u^{q^2} \\ 
	v & v^q & v^{q^2}
	\end{matrix} \right)= \alpha X^{q^2} + \beta X^q + \alpha^q X,
	\]
	where $\alpha=uv^q-vu^q$ and $\beta=u^{q^2}v-v^{q^2}u$. Hence the polynomial
	\[ g(X)= \overline{g}(X)/\alpha= X^{q^2}+\frac{\beta}{\alpha} X^q + \alpha^{q-1}X\]
	has exactly $q^2$ zeros. As every element in $\{c: \mathrm{Norm}_{q^{n}/q}(c)=1\}$ can be written as $\gamma^{q-1}$ for some $\gamma\in \F_{q^n}^*$, if we can show that
	\begin{equation}\label{eq:alpha=all}
		\{uv^q-vu^q: u,v\in\F_{q^n}^*  \}= \F_{q^n},
	\end{equation}
	then for any $b\in \F_{q^n}$ satisfying $\mathrm{Norm}_{q^{n}/q}(b)=1$, we can find $\alpha=uv^q-vu^q$ such that $b=\alpha^{q-1}$, whence there exists $a=\frac{\beta}{\alpha}$ such that the polynomial $X^{q^2}+a X^q + bX$ has $q^2$ roots.
	
	To prove \eqref{eq:alpha=all}, it is equivalent to show that for given $\alpha\in \F_{q^n}^*$, there exist $u$ and $v$ such that 
	\[\frac{u}{v} - \left(\frac{u}{v}\right)^q = \frac{\alpha}{v^{q+1}}.\]
	We consider the curve $\cC$ defined by
	\[Y-Y^q = \alpha X^{q+1}.\]
	It is clear that $\cC$ is absolutely irreducible. By Hasse-Weil bound, it has at least $q^n+1- q(q-1)\sqrt{q^n}$ rational points in $\PG(2, q^n)$. As $\cC$ only has one point $(0,1,0)$ at infinity, the total number of affine points is at least
	\[q^n-q(q-1)q^{n/2},\]
	which is always larger than $0$ if $n\ge 4$. Hence \eqref{eq:alpha=all} holds for $n\ge 4$.
	
	It is obvious that for $n=1,2$ the original question is meaningless. When $n=3$ and $q$ is even, it is easy to see $\gcd(q+1, q^n-1)=1$ which means the map $v\mapsto v^{q+1}$ is bijective. Thus, for every $\alpha\in \F_{q^n}$, we can fix the value of $u/v$ and choose $v$ such that 
	\[\alpha=v^{q+1}\left(\frac{u}{v} - \left(\frac{u}{v}\right)^q\right).\]

	When $n=3$ and $q$ is odd, we have $\gcd(q+1, q^n-1)=2$. Hence, by taking a non-square element $\gamma\in \F_q$, we have
	\[\left\{\frac{\alpha} {v^{q+1}}: v\in \F_{q^n}^*  \right\} 
	\cup \left\{\frac{\alpha \gamma} {v^{q+1}}: v\in \F_{q^n}^*  \right\}=\F_{q^n}^*. \]
	Hence,  for every $\alpha\in \F_{q^n}$, we can fix the value of $u/v$ and find $v$ such that 
	\[\alpha=v^{q+1}\left(\frac{u}{v} - \left(\frac{u}{v}\right)^q\right) ~\text{ or }~ \frac{v^{q+1}}{\gamma}\left(\frac{u}{v} - \left(\frac{u}{v}\right)^q\right).\]
	If the first case happens, we are done; otherwise we replace the value of $u/v$ by $u/(v\gamma)$ to get
	\[\alpha =v^{q+1}\left(\frac{u}{v\gamma} - \left(\frac{u}{v\gamma }\right)^q\right).\]
	Thus \eqref{eq:alpha=all} holds for $n=3$ and we finish the proof that $\mathrm{Norm}_{q^n/q}(b)\neq 1$ is a necessary condition for $bX + aX^{q} + X^{q^2}$ having at most $q$ zeros for every $a\in \F_{q^n}$.
	
	Next let us prove the corollary. Assume that $f$ is an exceptional scattered polynomial of index $1$ and degree $q^k$. By Theorem \ref{th:main_1_function}, when $n$ is large enough, we must have $k=2$ whence $f(X)=bX + X^{q^2}$ (Recall that we assume the coefficient of $X^{q^t}$ is always zero, where $t$ denotes the index of the exceptional scattered polynomials). By the first part of this proof, $\mathrm{Norm}_{q^n/q}(b)\neq 1$. In particular, when $q=2$, $b$ has to be $0$. As we always assume that the coefficient of $X$ in $f$ is nonzero, the only exceptional scattered polynomial of index $1$ for $q=2$ is $X$.
\end{proof}

\section{Concluding remarks}
In this paper, we have obtained a complete classification of exceptional scattered polynomial of index $1$. For those of index $0$, our classification results are complete for $q>5$. For $q=2,3,4,5$, Corollary \ref{coro:main_0_clear} shows that an exceptional scattered polynomial of index $0$ and degree $q^k$ have at most two extra terms $X^{q^{k-2}}$ and $X^{q^{k-1}}$. We leave the classification for $q=2,3,4,5$ as an open question.

For $t>1$, in general we do not have any complete classification result. However, by \cite[Theorem 10]{gow_galois_linear_2009} or \cite[Theorem 3.2]{lunardon_generalized_2015}, the classification for $t=1$ in Corollary \ref{coro:main_1_clear} implies that, for $s$ satisfying $\gcd(s,n)=1$, the $q^s$-polynomial $f=\sum_{j=0}^{k} a_i X^{q^{is}}\in \F_{q^n}[X]$ is an exceptional scattered polynomial of index $s$ if and only if $k\le 2$.

To obtain a classification result of exceptional scattered polynomials of index $t>1$, if we follow the approach in the proofs of Theorem \ref{th:main_0} and Theorem \ref{th:main}, the calculation becomes more involved and we also have to face the case in which $t>k$, where $q^k$ is the degree of $f$. It appears that our approach cannot give us a complete classification for these cases. Hence we leave the classification of exceptional scattered polynomials of index $t>1$ as the second open question.

\section*{Acknowledgment}
Yue Zhou is partially supported by the National Natural Science Foundation of China (No.\ 11401579). The work of Daniele Bartoli was supported in part by Ministry for Education, University and Research of Italy (MIUR) (Project PRIN 2012 ``Geometrie di Galois e strutture di incidenza'') and by the Italian National Group for Algebraic and Geometric Structures and their Applications (GNSAGA - INdAM).

%
\end{document}